\def\softd{{\leavevmode\setbox1=\hbox{d}%
\hbox to 1.05\wd1{d\kern-0.4ex{\char039}\hss}}}
\def\softt{{\leavevmode\setbox1=\hbox{t}%
\hbox to \wd1{t\kern-0.6ex{\char039}\hss}}}
\def\softl{l\kern-0.45ex\raise0.1ex\hbox{'}\kern-0.10ex}
\def\softL{L\kern-0.8ex\raise0.1ex\hbox{'}\kern0.1ex}

\documentclass[12ptt,oneside,reqno]{amsart}
\usepackage{a4}
\usepackage{amssymb}
\usepackage{amsmath}
\usepackage{amsthm}
\usepackage{bm}
\usepackage{epic,eepic}
\usepackage[mathscr]{euscript}
\usepackage{enumerate}
 \usepackage{stmaryrd}

\usepackage{graphicx}

\usepackage{lineno}
\usepackage{comment}

\def\Ekviv{\hbox{$\Longleftrightarrow$}}
\newtheorem{lemma}{Lemma}[section]
\newtheorem{theorem}[lemma]{Theorem}
\newtheorem{proposition}{Proposition}[section]

\newcommand{\cal}{\mathcal}

\newcommand{\Ecal}{{\cal E}}
\newcommand{\bmdot}{\bm{\cdot}}
\newcommand{\bmto}{\bm{\to}}
\newcommand{\bmleadsto}{\bm{\leadsto}}

\renewcommand{\phi}{\varphi}

\newtheorem{corollary}[lemma]{Corollary}
\newtheorem{definition}[lemma]{Definition}

\newcommand{\mathify}[1]{\ifmmode{#1}\else\mbox{$#1$}\fi}
\newcommand{\cle}{\mathrel{\raisebox{0ex}[1.5ex][.4ex]{\raisebox{.25ex}{$\sqsubset$}\kern-1.75ex\rai
sebox{-.9ex}{$\sim$}}}}
\newcommand{\cge}{\mathrel{\raisebox{0ex}[1.5ex][.4ex]{\raisebox{.25ex}{$\sqsupset$}\kern-1.75ex\ra
isebox{-.9ex}{$\sim$}}}}

\newcommand{\ua}{\mathord{\uparrow}}

\newcommand{\impl}{\mbox{$
    \mathop{-\raise.8pt\hbox{\mathsurround=0pt$\!\scriptstyle\circ$}}$}}

\def\squig{\leadsto}
\def\logoesf{
\begin{tabular}{l l}
\begin{tabular}{c}
{Supported by}\\
\phantom{\huge X}
\end{tabular}& \ \resizebox{8.58cm}{!}{\includegraphics{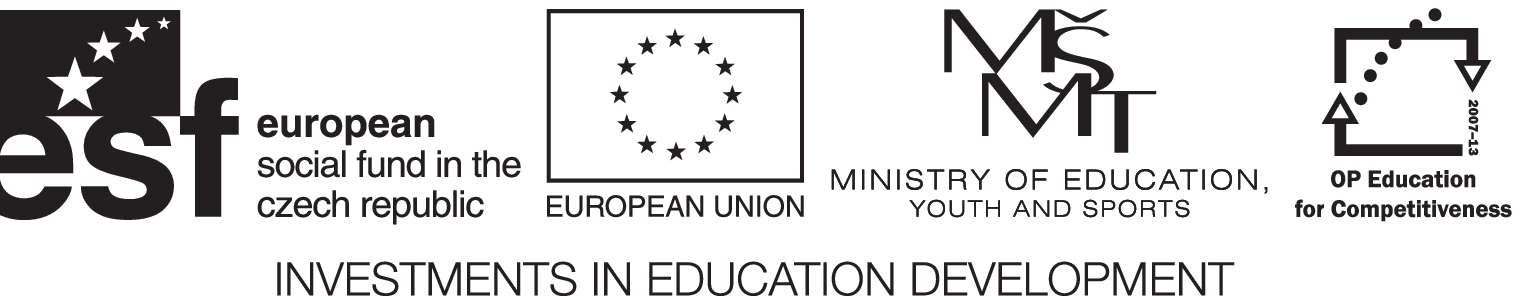}}
\end{tabular}}

\begin{document}

   \title{Filters on some classes of quantum B-algebras}
    \author{Michal Botur, Jan Paseka}
\address{Palack\' y University Olomouc, Faculty of Sciences, t\v r. 17.listopadu 1192/12, Olomouc 771 46, Czech Republic}
\email{michal.botur@upol.cz}
\address{Department of Mathematics and Statistics, Faculty of Science,
Masaryk University, {Kotl\'a\v r{}sk\' a\ 2}, 611~37 Brno, 
Czech Republic}

\email{paseka@math.muni.cz}

\thanks{Both authors acknowledge the support by Austrian Science Fund (FWF): project I 1923-N25 
and by ESF Project CZ.1.07/2.3.00/20.0051
Algebraic methods in Quantum Logic of the Masaryk University.\\
\logoesf%
}
 \date{}

\maketitle

\begin{abstract} In this paper, we continue the study of quantum B-algebras 
with emphasis on filters on integral quantum B-algebras. We then study filters in the setting of pseudo-hoops. 
First, we establish an embedding of a cartesion product of polars of a pseudo-hoop  into itself. 
Second, we give sufficient conditions for a pseudohoop to be subdirectly reducible. 
We also  extend the result of Kondo and Turunen to the setting of 
noncommutative residuated $\vee$-semilattices that,  if prime filters and $\vee$-prime filters 
of a  residuated $\vee$-semilattice $A$ coincide,
then $A$ must be a pseudo MTL-algebra.
\end{abstract}

\parskip7pt plus2pt minus1pt

{\small

\parskip7pt plus2pt minus1pt


}

\parskip7pt plus2pt minus1pt

\section{Introduction}

The term {\em quantale} was suggested by
C.J.~Mulvey at the Oberwolfach Category Meeting
\ (see \cite{mulvey}) as
 "a quantization" of the term {\em locale}.
Locales  form an order-theoretic counterpart of
topological spaces and are therefore able to
describe commutative C$^{*}$-algebras. The main aim of
C.J.~Mulvey has been to find a substitute of locales
which could play the same r{\^o}le for general C$^{*}$-algebras 
to establish a generalized Gelfand--Naimark duality
for all C*-algebras and study non-commutative topology. 
Quantales are also applied in linear and other substructural logics 
and automaton theory. An important moment in the development of the theory of quantales
was the realization that quantales give a semantics for
propositional linear logic in the same way as Boolean
algebras give a semantics for classical propositional
logic (see \cite{girard}). Quantales arise naturally as lattices of ideals, subgroups,
or other suitable substructures of algebras, and then they are
called \emph{spectra}. 

By definition, a {\em quantale}
is a complete lattice $Q$ with an associative multiplication 
$\cdot$ that distributes over arbitrary joins.  By the completeness of $Q$, there are left and right
adjoint operations (called {\em residuals}) $\to$ and $\leadsto$ of  $\cdot$ such that 
$$
\phantom{xxxxx}x\leq y\to z\ \text{if and only if}\ 
x\cdot y\leq z\ \text{if and only if}\  
y\leq x\leadsto z. \phantom{xxxxx}\text{\rm (R)}
$$

Note that, as was mentioned in \cite{rump} and \cite{ruya}, in any quantale the following conditions are 
satisfied:
$$
\begin{array}{c}
\begin{array}{r c l}
y \to z &\leq& (x \to y) \to (x \to z)\\
y \leadsto z &\leq& (x \leadsto y) \leadsto (x \leadsto z)\\
\end{array}\\
\phantom{xxxxx}\phantom{xxxxx}\phantom{xxxxx}y \leq z  
\implies  x \to y \leq x \to  z\phantom{xxxxx}\phantom{xxxxx}\phantom{xxxxx}\text{\rm (QB)}\\
x \leq y \to z\  \Ekviv\  y \leq x \leadsto z.
\end{array}
$$

This lead Rump and Yang in \cite{ruya} to introduce {\em quantum B-algebras} which 
formalize the implicational part of the logic of quantales. Note that quantum B-algebras 
encompass pseudo-BCK algebras, partially ordered monoids with two residuals satisfying 
(R) and generalized pseudoeffect algebras. Moreover, in \cite{ruya} they established 
a one-to-one correspondence between quantum B-algebras and so-called 
logical quantales.

In this paper, we continue the study of quantum B-algebras 
from \cite{ruya,rump} with emphasis on filters 
on integral quantum B-algebras. Namely, the filter theory of logical algebras (see e.g. 
\cite{gasse,zhu}) 
plays an  significant role in studying these algebras and the completeness of the 
corresponding non-classical logics. It is natural to consider filters of algebras which
are corresponding to congruences and to investigate quotient
algebras by such filters. Recall that, from a logical point of view, filters correspond 
to sets of provable formulas. 

During the last decade study of many-valued reasoning a lot of  noncommutative generalizations,  
which generalize MV-algebras developed by C.C. Chang \cite{Chan},  were introduced. Let 
us mention  for example pseudo MV-algebras \cite{GeIo} (independently introduced also 
in \cite{Rac} as generalized MV-algebras), pseudo BL-algebras \cite{DGI1, DGI2} and 
pseudo-hoops,\cite{GLP}. We recall that pseudo BL-algebras are also a noncommutative generalization 
of P. H\'ajek's BL-algebras: a variety that is an algebraic counterpart of fuzzy
logic \cite{Haj}. Therefore, a pseudo BL-algebra is an algebraic
presentation of a non-commutative generalization of fuzzy logic. These structures are studied 
also in the area of quantum structures, see \cite{MNRP}.

However, as it was recently recognized,  many of these notions have a very close connections 
with notions introduced already by B. Bosbach in his pioneering papers on various classes of 
semigroups: among others he introduced complementary semigroups (today known as pseudo-hoops). 
A deep investigation of these structures can be found in his papers \cite{Bos1, Bos2}; more information 
is available in his recent papers \cite{Bos3,Bos4}. Nowadays, all these structures can be also studied 
under one common roof, as residuated lattices, \cite{GaTs}. The theory of filters, representations 
and normal-valued basic pseudo-hoops was studied in \cite{BDK}. Now all these structures are  intensively studied by many experts (see  \cite{JiMo},\cite{Dvu4}, \cite{AgMo},  \cite{DGK}).

The paper is organized as follows. 
After introducing several necessary algebraic concepts as quantale or quantum B-algebra in Section \ref{basicn} 
we introduce following  \cite{ruya,rump} a multiplication $\cdot$ on the complete lattice $U(A)$ of upper subsets 
of a quantum B-algebra $A$ that makes $U(A)$ a quantale.  Filters 
on an integral quantum B-algebra $A$ are exactly idempotent elements of $U(A)$. 

In  Section \ref{filtersb} we show that, for a filter $F$ of an integral quantum B-algebra $A$, the 
set $U(F)$ of upper subsets of  the filter $F$ is a subquantale of the quantale $U(A)$ using a map 
$\mu_F:U(A) \to U(A)$. Further, we establish basic properties of the map $\mu_F$.

In Section \ref{michal} we study filters in the setting of pseudo-hoops. First, we establish 
an embedding of a cartesion product of polars of a pseudo-hoop  into itself. 
Second, we give sufficient conditions for a pseudohoop to be subdirectly reducible. 

In Section \ref{filterkondo} we extend the result of Kondo and Turunen (see \cite{kondotur}) to the setting of 
noncommutative residuated $\vee$-semilattices that,   if prime filters and $\vee$-prime filters 
of a  residuated $\vee$-semilattice $A$ coincide,
then $A$ must be a pseudo MTL-algebra.

The terminology and symbols used here
coincide in general with those used in \cite{ono}.

\section{Basic notions}\label{basicn}

Now, let us proceed by stating the definitions, some of them well known.

 A {\em quantum B-algebra} is a poset
$A$ with two binary operations $\to$ and $\leadsto$ satisfying conditions
$$
\begin{array}{c}
\begin{array}{r c l}
y \to z &\leq& (x \to y) \to (x \to z)\\
y \leadsto z &\leq& (x \leadsto y) \leadsto (x \leadsto z)\\
\end{array}\\
y \leq z  \implies  x \to y \leq x \to  z
\end{array}
$$
and the equivalence
$$
x \leq y \to z\  \Ekviv\  y \leq x \leadsto z
$$
for all $x,y, z\in A$. A quantum B-algebra $A$ is {\em unital}
if $A$ admits an element $u$, the {\em unit element}, which
satisfies $u \to x = u \leadsto x = x$ for all $ x \in A$.
A unit element is unique.  

The unit element  reduces the relation $\leq$ to
the operations  $\to$ and $\leadsto$:\\
$$
x \leq y  \  \Ekviv\  u \leq x \to y   \  \Ekviv\  u \leq x \leadsto y.
$$
Thus, if the unit element $u$ is  the greatest element of  $A$, the relation
$x \leq y$ just means that $x \to y$ is true.

 An {\em integral quantum B-algebra}  or a {\em pseudo BCK-algebra} is a unital quantum B-algebra $A$ such that
$u$ is the top element of $A$, i.e. $u=1$. 

 A {\em residuated poset} is a partially ordered semigroup 
$(A; \cdot)$ with two
binary operations $\to$ and $\leadsto$ which satisfy 
$$
x \cdot y \leq z\ \Ekviv \    x \leq y \to z  \  \Ekviv\  y \leq x \leadsto z.
$$
Every residuated poset is a quantum B-algebra. 
A residuated poset $(A; \cdot, \to, \leadsto, \leq)$ is called 
{\em 2-sided} if $x\cdot y\leq x$ and $x\cdot y\leq y$ 
for all $x, y\in A$. We say that residuated poset $(A; \cdot, \to, \leadsto, \leq)$ is 
a {\em residuated $\vee$-semilattice} if $(A;\vee)$ is a semilattice with respect to the order $\leq$.

A {\it quantale\/} is a complete lattice $Q$ with an associative
binary multiplication satisfying
$$
x\cdot\bigvee\limits_{i\in I}
x_i=\bigvee\limits_{i\in I}x\cdot
x_i\ \ \hbox{and}\ \ (\bigvee\limits_{i\in I}x_i)\cdot
x=\bigvee\limits_{i\in I}x_i\cdot x
$$
for all $x,\,x_i\in Q,\,i\in I$ ($I$ is a set).

An element  $x\in Q$  is called {\it idempotent\/} if $x\cdot x=x$. 
$1$ denotes the greatest element
of $Q$, $0$ is the smallest element of $Q$.
 The set of all idempotent elements of a quantale $Q$ is
denoted by $ \Ecal{}(Q)$. We shall say that
a quantale $Q$ is said to be
idempotent if
$Q=\Ecal{}(Q)$.
In the event that $Q$ has only
one element we shall speak about a {\em trivial quantale}.

Since the operators $a\cdot -$ and $-\cdot b: Q\to Q$, $a, b\in Q$ preserve
arbitrary suprema they have right adjoints. We
shall denote them by $a\leadsto{}-$ and $b\to -$ respectively. 

Let  $a,b,c,a_i\in Q$. Then
$$a\leadsto (b\to c) = b\to (a\leadsto c), $$
\begin{align*}
a\to (b\to c) &= (ab)\to c, & b\leadsto (a\leadsto c)&= (ab)\leadsto c,\\
\left(\bigvee a_i\right)\to c &= \bigwedge(a_i\to c), & \left(\bigvee a_i\right)\leadsto c &= \bigwedge(a_i\leadsto c).
\end{align*}

Evidently, any quantale is a residuated poset and hence a quantum B-algebra.

Since every quantale Q is a complete lattice, the {\em inverse residuals}
$$
\begin{array}{r c l}
a   \rightarrowtriangle  b &:= &\bigwedge\{x \in Q \mid x \cdot  a \geq  b\}\\
a \rightarrowtail b &:= &\bigwedge\{x \in Q \mid a \cdot  x \geq  b\}
\end{array}
$$

are well-defined, too.

A non-zero element $c \in Q$ is  {\em balanced} if it satifies 
$$
c\cdot\bigwedge\limits_{i\in I}
x_i=\bigwedge\limits_{i\in I}c\cdot
x_i\ \ \hbox{and}\ \ (\bigwedge\limits_{i\in I}x_i)\cdot
c=\bigwedge\limits_{i\in I}x_i\cdot c
$$
for all $x_i\in Q,\,i\in I$ ($I$ is a set).

An element $c$ of a complete lattice $L$ is said to be 
{\em supercompact} if for any non-empty subset $X \subseteq L$, the inequality
$c \leq \bigvee X$ implies that $c \leq x$ for some $x \in X$.

For every quantum B-algebra $A$, the upper sets $X\subseteq A$
(i. e.  the subsets $X$ with $ a \geq  b \in X$ implies $a \in X$) can be made into a quantale 
$U(A)$
by defining
$$X\cdot Y := \{a \in  A \mid (\exists y\in Y) (y\to a)\in X\}.$$

It can be shown \cite{ruya} that this gives an associative multiplication which distributes
over set-theoretic joins. Therefore, 
$$\begin{array}{c l}&X\leadsto Z :=\{y\in A \mid (\forall x\in X)(\forall z\in A)(x\leadsto z\geq y\ \implies\ z\in Z) \}\\
\quad \text{and} \quad&\\
&Y\to Z := \{x\in A \mid (\forall y\in Y)(\forall z\in A)(y\to z\geq x\ \implies z\in Z)\}.
\end{array}$$

If $A$ is a residuated poset then 
$$X\cdot Y = \{a \in  A \mid (\exists x\in X)(\exists y\in Y)(x\cdot y\leq a)\},$$
$$\begin{array}{c l}&X\leadsto Z :=\{y\in A \mid (\forall x\in X)(\forall z\in A)(x\cdot y\leq z\  \implies\ z\in Z) \}\\
\quad \text{and} \quad&\\
&Y\to Z := \{x\in A \mid (\forall y\in Y)(\forall z\in A)(x\cdot y\leq z \ \implies\ z\in Z)\}.
\end{array}$$

In this case, for any $ n \in {\mathbb N}, n\geq 1$ and any  $x\in  A$ we put 
$x^{1} = x$ and  $x^{n+1} = x^{n}\cdot  x = x\cdot  x^{n}$.

A {\em filter} $F$ of a quantum B-algebra  $A$ is a non-empty set 
$F\in U(A)$ such that $F\cdot F\subseteq F$. Note that this is equivalent 
with $z\in A, y\in F$, $y\to z\in F$ yields $z\in F$ and that $F$ is a non-empty upper subset of $A$. Recall also that 
any non-empty set $F\in U(A)$ that is idempotent is a filter. We denote by ${\mathcal F}(A)$ the set of all filters of $A$. Recall that 
any non-empty intersection of filters is again a filter and any directed union of filters is a filter.

For every non-empty subset $X \subseteq  A$, the smallest filter of  $A$ containing  $X$ 
(i.e., the intersection of all filters $F \in {\mathcal F}(A)$ such that $X \subseteq  F$)  is called 
the {\em filter generated by} $X$ and will be denoted by [X).

If $A$ is a residuated poset then  
$$[X) = \{y \in A \mid y\geq  x_1 \cdot x_2\cdot \dots  \cdot x_n\ \text{for some}\  n\in {\mathbb N},  n\geq 1\  \text{and}\  x_1, x_2, \dots, x_n \in  X\}.$$ 
Moreover, for an 2-sided residuated poset $A$ such that  $F$ is a filter and $a\in A$, $a\notin F$ we have that 
$$
\begin{array}{r c l l}
[F\cup \{a\})&=&
 \{y \in A \mid& y\geq  x_1 \cdot a\cdot x_2 \cdot a \cdot \dots \cdot a \cdot  x_n\ \text{for some}\  n\in {\mathbb N},  n\geq 1\\ 
& & &\text{and}\  x_1, x_2, \dots, x_n \in  F\}.\\
\end{array}
$$

 Furthermore, the set of supercompact elements of  $U(A)$
coincides with the image of the embedding 
$A \hookrightarrow U(A)$ given by $x \mapsto \ua{}x$, and every
balanced element of $U(A)$ is supercompact.  
Similarly, the  embedding 
$U(A) \hookrightarrow  U(U(A))$  is given by $X \mapsto
\ua{}X$.
A quantale $Q$ is called {\it unital\/} if
there is an element $e\in Q$ such that
$$
e\cdot a = a = a\cdot e
$$
\noindent
for all $a\in Q$.

A {\it  subquantale\/} $S$ of $Q$ is a subset of
$Q$ closed under all  suprema and  $\cdot$\,. 
 $S$ is said to be a trivial
subquantale if $S=\{ 0 \}$ or $S=Q$.
A quantic (nucleus) conucleus on $Q$ is a (closure)
coclosure operator $g$ such that $g(a)\cdot g(b) \leq  g(a\cdot b)$ for all $a, b \in Q$.
A quantic conucleus g is said to be trivial if $g(a)=a$  or $g(a)=0$ for all $a \in Q$. 

If $g$ is a quantic  conucleus on $Q$, then
$Q_g = \{ a \in Q \mid g(a)=a \}$ is a 
subquantale of $Q$. Moreover, if $S$ is any 
subquantale of $Q$, then $S=Q_g$ for some quantic  
conucleus $g$.

\section{Filters in Quantum B-algebras}\label{filtersb}

In this section  we show that, for a filter $F$ of an integral quantum B-algebra $A$, the 
set $U(F)$ of upper subsets of  the filter $F$ is a subquantale of the quantale $U(A)$ using a map 
$\mu_F:U(A) \to U(A)$. Further, we establish basic properties of the map $\mu_F$.

Let us put, for any $F\in U(A)$ and  $X\in U(A)$, 
$\mu_{F}(X)=F\cap X$. Then, for any $F\in U(A)$,  
$\mu_F:U(A)\to U(A)$ is an order preserving idempotent map. 
Evidently, if $X, Y\in U(A)$, $X\subseteq Y$ then 
$\mu_F(X)=\mu_{F}(Y)\cap X=\mu_{\mu_{F}(Y)}(X)$ and 
$\mu_{F}(X)=F\cap X=F\cap (F\cap X)=\mu_{F}(\mu_{F}(X))$.

\begin{lemma}\label{muf} Let $A$ be a quantum B-algebra, $X, Y, F\in  U(A)$, 
$F$ a filter of $A$. 
Then $\mu_F(X) \cdot \mu_F(Y)\subseteq \mu_F(X\cdot Y)$. Moreover 
$\mu_F$ is a conucleus on $U(A)$ and the set 
$U(F)=\{ U\in U(A)\mid U\subseteq F\}=\{\mu_F(X)\mid X\in U(A)\}$ equipped with the 
multiplication $\cdot_F=\cdot /U(F)$ is a subquantale of $U(A)$.
\end{lemma}
\begin{proof} Assume that $a\in A$  and $a\in \mu_F(X) \cdot \mu_F(Y)$. Then 
there is $y\in F\cap Y$ such that 
$y\to a\in F\cap X$. It follows that  $a\in F\cdot F\subseteq F$ and $a\in X\cdot Y$. Therefore 
$a\in F\cap (X\cdot Y)$. The remaining part is evident.
\end{proof}

In what follows let $A$ be an integral quantum B-algebra. 
Note also that,  for any $F\in U(A)$ and  $X\in U(A)$ such that 
$1\in X\cap F$, 
 $F\cdot X\supseteq F\cup X$, $1\in \mu_{F}(X)\to X$ and $1\in \mu_{F}(X)\leadsto X$. 
In particular, $F$ is an filter if and only if $F\cdot F=F$ and $1\in F$. Moreover, 
for  any $X\in U(A)$, $X=\{1\}\cdot X=X\cdot \{1\}$.

\begin{proposition}\label{mufid} Let $A$ be an integral quantum B-algebra, $X, Y, F\in  U(A)$, 
$F$ a filter of $A$. 
Then the following holds: 
\begin{enumerate} 
\item $\mu_F(X)=\mu_{\mu_F(X)\cdot \mu_F(X)}(X)$; 
\item If $1\in \mu_{F}(X)\leadsto (\mu_{F}(X)\to X)$ then $\mu_{F}(X)\cdot \mu_F( \mu_{F}(X)\leadsto (\mu_{F}(X)\to X))\cdot \mu_{F}(X)=\mu_F(X)$; 
\item $\mu_F(X)$ is a filter of $A$ iff $1\in \mu_F(X)\cap \mu_{F}(X)\leadsto (\mu_{F}(X)\to X)$.
\end{enumerate}
\end{proposition}
\begin{proof} (1)  Let 
$z\in \mu_{\mu_F(X)\cdot \mu_F(X)}(X)=(\mu_F(X)\cdot \mu_F(X))\cap X$. Then 
$z\in X$ and $z\in (F\cap X)\cdot (F\cap X)\subseteq F$. Hence $z\in \mu_F(X)$. 
Conversely, let $z\in \mu_F(X)$. Then $1\in \mu_F(X)$ and  
$z\in [z)\cdot [1)\subseteq (F\cap X)\cdot (F\cap X)$ 
and $z\in X$. It follows that $z\in \mu_{\mu_F(X)\cdot \mu_F(X)}(X)$.\\
(2) Evidently, $ \mu_{F}(X)\subseteq  \mu_{F}(X)\cdot \mu_F( \mu_{F}(X)\leadsto (\mu_{F}(X)\to X))\cdot \mu_{F}(X)$. 
To show the converse direction let us compute:
$$
\begin{array}{l}
\mu_{F}(X)\cdot \mu_F( \mu_{F}(X)\leadsto (\mu_{F}(X)\to X))\cdot \mu_{F}(X) \subseteq \\
\mu_{F}(X)\cdot (\mu_{F}(X)\leadsto (\mu_{F}(X)\to X))\cdot \mu_{F}(X)) \subseteq 
\mu_{F}(X)\cdot (\mu_{F}(X)\to X)\subseteq X.
\end{array}
$$
Since $ \mu_{F}(X)\subseteq F$ and $\mu_F( \mu_{F}(X)\leadsto (\mu_{F}(X)\to X))\subseteq F$ we get 
that  $\mu_{F}(X)\cdot \mu_F( \mu_{F}(X)\leadsto (\mu_{F}(X)\to X))\cdot \mu_{F}(X) \subseteq F$. 
It follows that 
$$
\mu_{F}(X)\cdot \mu_F( \mu_{F}(X)\leadsto (\mu_{F}(X)\to X))\cdot \mu_{F}(X) \subseteq \mu_{F}(X).$$

\noindent{}(3) Evidently, if $1\in \mu_{F}(X)\leadsto (\mu_{F}(X)\to X)$ then 
$\{1\}\subseteq \mu_{F}(X)\leadsto (\mu_{F}(X)\to X)$. This yields that 
  $\mu_F(X)\cdot \mu_F(X) =\mu_F(X)\cdot \{1\}\cdot  \mu_F(X) \subseteq  
\mu_F(X) \cdot (\mu_{F}(X)\leadsto (\mu_{F}(X)\to X))\cdot \mu_F(X) \subseteq \mu_F(X)$. 
Since  $1\in \mu_F(X)$   we get that $\mu_F(X)$ is a filter of $A$.

Conversely, let $\mu_F(X)$ be a filter of $A$. Then 
  $\mu_F(X)\cdot \mu_F(X) \subseteq  \mu_F(X) \subseteq X$. It follows that 
$\mu_F(X)\cdot \mu_F(X) \subseteq  \mu_F(X) \subseteq \mu_{F}(X)\to X$.  
By the same reasoning $ 1\in \mu_F(X) \subseteq \mu_{F}(X)\leadsto (\mu_{F}(X)\to X)$.
\end{proof}

\begin{proposition}\label{mufc} Let $A$ be an integral quantum B-algebra, $X, Y, F\in  U(A)$, $F$ a filter of $A$ 
and $1\in X\cap F$.
Then the following holds 
\begin{enumerate}
\item $\mu_F(X)\cdot \mu_F( \mu_{F}(X)\leadsto X)  = \mu_F(X)=\mu_F(\mu_{F}(X)\to X)\cdot \mu_{F}(X)$; 
\item $\mu_F( \mu_{F}(X)\to X)\to (\mu_{F}(X)\to X)=\mu_F(X)\to X$;
\item $\mu_F( \mu_{F}(X)\leadsto X)\leadsto (\mu_{F}(X)\leadsto X)=\mu_F(X)\leadsto X$; 
\item $\mu_F( \mu_{F}(X)\to X) \cdot (\mu_{F}(X)\to X)=\mu_{F}(X)\to X$;
\item $( \mu_{F}(X)\leadsto X) \cdot \mu_F(\mu_{F}(X)\leadsto X)=\mu_{F}(X)\leadsto X$;
\item $\mu_F(\mu_{F}(X)\to X)=\mu_F( \mu_{F}(X)\to X) \cdot \mu_F(\mu_{F}(X)\to X)$ 
and $\mu_F(\mu_{F}(X)\to X)$ is a filter of $A$ whenever $1\in \mu_F(\mu_{F}(X)\to X)$.
\item $\mu_F(\mu_{F}(X)\leadsto X)=\mu_F( \mu_{F}(X)\leadsto X) \cdot \mu_F(\mu_{F}(X)\leadsto X)$ 
and $\mu_F(\mu_{F}(X)\leadsto X)$ is a filter of $A$ whenever $1\in \mu_F(\mu_{F}(X)\leadsto X)$.
\end{enumerate}
\end{proposition}
\begin{proof} (1) Since $1\in \mu_F( \mu_{F}(X)\leadsto X) \cap \mu_F(X)$ we obtain that 
$\mu_F(X)\cdot\mu_F( \mu_{F}(X)\leadsto X) \supseteq  \mu_F(X)$. 
Conversely, we have $\mu_F(X)\cdot\mu_F( \mu_{F}(X)\leadsto X)\subseteq %
( \mu_{F}(X)\leadsto X) \cdot \mu_F(X)\subseteq X$ and 
$\mu_F(X)\cdot\mu_F( \mu_{F}(X)\leadsto X)\subseteq F\cdot F\subseteq F$. 
It follows that $\mu_F(X)\cdot\mu_F( \mu_{F}(X)\leadsto X)\subseteq \mu_F(X)$. 
The remaining part follows by analogous considerations.\\
(2)  $\mu_F( \mu_{F}(X)\to X)\to (\mu_{F}(X)\to X)= (\mu_F( \mu_{F}(X)\to X)\cdot  \mu_{F}(X))\to X=\mu_F(X)\to X$.\\
(3) As in (2).\\
(4) Since $1\in \mu_F( \mu_{F}(X)\to X) \cap (\mu_{F}(X)\to X)$ we get  
$(\mu_{F}(X)\to X)\cdot\mu_F( \mu_{F}(X)\to X)  \supseteq \mu_{F}(X)\to X$. To prove the converse direction let us compute:
$$
\begin{array}{r @{}l}
(\mu_{F}(X)\to X)\cdot &\mu_F( \mu_{F}(X)\to X) =\\
&(\mu_F( \mu_{F}(X)\to X)\to (\mu_{F}(X)\to X))\cdot \mu_F( \mu_{F}(X)\to X) \subseteq\\
&\mu_{F}(X)\to X.\\
\end{array}
$$
Whence $ (\mu_{F}(X)\to X)\cdot \mu_F( \mu_{F}(X)\to X) =\mu_{F}(X)\to X$.\\
(5) As in (4).\\
(6) Evidently, $\mu_F(\mu_{F}(X)\to X)\supseteq\mu_F( \mu_{F}(X)\to X) \cdot \mu_F(\mu_{F}(X)\to X)$.
Conversely, we 
have 
$$
\begin{array}{rcl}
\mu_F( \mu_{F}(X)\to X) \cdot \mu_F(\mu_{F}(X)\to X)%
&\subseteq& \mu_F( \mu_{F}(X)\to X) \cdot (\mu_{F}(X)\to X)=\\
&&\mu_{F}(X)\to X
\end{array}
$$
and $\mu_F( \mu_{F}(X)\to X) \cdot \mu_F(\mu_{F}(X)\to X)\subseteq F\cdot F\subseteq F$. Consequently, we obtain that 
$1\in \mu_F(\mu_{F}(X)\to X)=\mu_F( \mu_{F}(X)\to X) \cdot \mu_F(\mu_{F}(X)\to X)$. 
It follows that 
$\mu_F(\mu_{F}(X)\to X)$ is a filter of $A$.\\
(7) As in (6).
\end{proof}

\section{Filters on pseudo-hoops}

\label{michal}

In the present section  we study filters in the setting of pseudo-hoops. First, we establish 
an embedding of a cartesion product of polars of a pseudo-hoop ${\mathbf A}$ into ${\mathbf A}$. 
Second, we give sufficient conditions for a pseudohoop to be subdirectly reducible.

We recall that according to \cite{GLP}, a \textit{pseudo-hoop} is
an algebra $\mathbf M=(M; \cdot, \to,\squig,1)$ of type $\langle 2,2,2,0
\rangle$ such that, for all $x,y,z \in M,$

\begin{enumerate}

\item[{\rm (i)}]   $x\cdot 1 = x = 1 \cdot x;$

 \item[{\rm (ii)}] $x\to x = 1 = x\squig x;$

\item[{\rm (iii)}] $(x\cdot y) \to z = x \to (y\to z);$

 \item[{\rm (iv)}] $(x \cdot y) \squig z = y \squig
(x\squig z);$

 \item[{\rm (v)}] $(x\to y) \cdot x= (y\to x)\cdot y =
x\cdot (x\squig y) = y \cdot (y \squig x)$ (divisibility).

\end{enumerate}

It can be easily checked that any pseudo-hoop is a residuated poset, see e.g. \cite[Proposition 2.1]{ciungu}.

If $\cdot$ is commutative (equivalently $ \to = \squig$), $\mathbf M$ is
said to be a \textit{hoop}.  If we set $x \le y$ iff $x \to y=1$
(this is equivalent to $x \squig y =1$), then $\le$ is a partial
order such that $x\wedge y = (x\to y)\cdot x=x\cdot (x\leadsto y)$ and $\mathbf M$ is a
$\wedge$-semilattice.

We say that a pseudo-hoop $\mathbf M$

\begin{enumerate}

\item[(i)] is  {\it bounded} if there is a least element $0,$
otherwise, $\mathbf M$ is {\it unbounded},

\item[(ii)] satisfies \textit{prelinearity} if, given $x,y \in M,$
$(x\to y)\vee (y\to x)$ and $(x\squig y)\vee (y\squig x)$ are
defined in $\mathbf M$ and they are equal $1,$

\item[(iii)] is \textit{cancellative} if $x\cdot y=x\cdot z$ and $s\cdot x= t
\cdot x$ imply $y= z$ and $s= t,$

\item[(iv)] is a \textit{pseudo BL-algebra}  if $\mathbf M$ is a bounded
lattice satisfying  prelinearity.
\end{enumerate}

For a pseudo BL-algebra, we define $x^-=x\to 0$ and $x^\sim =
x\squig 0.$ A pseudo BL-algebra is said to be a pseudo MV-algebra if
$x^{-\sim}=x=x^{\sim -}$ for every $x \in M.$

From (v) of the definition of pseudo-hoops we have that a pseudo
hoop is cancellative iff $x\cdot y\le x\cdot z$ and $s\cdot x\le t
\cdot x$ imply $y\le z$ and $s\le t.$

Let us have a pseudo-hoop $\mathbf A=(A;\cdot,\to,\leadsto, 1).$ Then we for any 
set $M\subseteq A$ define the set $M^\perp=\{x\in A\mid x\vee y = 1\mbox{ for any } y\in M\}.$ 
One can easily check that following conditions hold:
\begin{itemize}
\item[1)] $M\subseteq N$ yields $N^\perp\subseteq M^\perp,$
\item[2)] $M\subseteq M^{\perp\perp},$
\item[3)] $M^\perp=M^{\perp\perp\perp}.$
\end{itemize}
Consequently, $^{\perp\perp}$ is a closure operator on the subsets of $A.$ 
If $x_1\vee y=x_2\vee y=1$ then we can compute 
$x_1x_2\vee y = x_1x_2\vee x_1y\vee y = x_1(x_2\vee y)\vee y=x_1\vee y=1.$ Thus, 
the set $M^\perp$ is a filter for any $M\subseteq A.$

\begin{lemma}\label{MB-Lemma1}
Let $\mathbf A=(A;\cdot,\to,\leadsto, 1)$ be a pseudo-hoop 
and let $x,y\in A$ be such that $x\vee y=1$ then $x\cdot y = x\wedge y=y\cdot x$ and $x\to y= x\leadsto y= y.$ 
\end{lemma}
\begin{proof} Assume that $x,y\in A$ are such that $x\vee y=1$. Then, for any $z\in A$, we have 
$z\leq y$ iff $z\leq 1\to y$ iff  $z\leq (x\vee y)\to y$ iff  $z\cdot (x\vee y)\leq y$ iff 
 $z\cdot x\vee z\cdot y\leq y$ iff  $z\cdot x\leq y$ iff $z\leq x\to y$. It follows that 
$y=x\to y$. By symmetry, $x=y\to x$.

Due to divisibility  we can 
compute $y\cdot x = (x\to y)\cdot x = x\wedge y =%
(y\to x)\cdot  y= x\cdot y .$
\end{proof}

\begin{theorem}\label{mb_theor1}
Let $\mathbf A=(A;\cdot,\to,\leadsto, 1)$ be a pseudo-hoop and let us 
have any set $M\subseteq A.$ Then the mapping $f:M^\perp\times M^{\perp\perp}\longrightarrow A$ 
defined by $f(x,y)=x\wedge y$ is a embedding from $\mathbf{M^{\perp}\times M^{\perp\perp}}$ to $\mathbf A.$ 
\end{theorem}
\begin{proof}
If $f(x_1,y_1)=f(x_2,y_2)$ for any 
$(x_1,y_1), (x_2,y_2)\in M^\perp\times M^{\perp\perp}$ 
then $x_1\wedge y_1=x_1\cdot y_1=y_1\cdot x_1=x_2\wedge y_2=x_2\cdot y_2=y_2\cdot x_2$ and also 
$x_i\vee y_j=1$ for any $i,j\in\{1,2\}.$ 

Firstly we prove that the mapping $f$ is a injection. We can compute 
$y_1=y_1\cdot 1= y_1\cdot (x_1\vee y_2)=y_1\cdot x_1\vee y_1\cdot y_2=%
x_2\cdot y_2\vee y_1\cdot y_2=(x_2\vee y_1)\cdot y_2=1\cdot y_2=y_2.$ Thus $y_2= y_1.$ 
By symmetry,  $x_1=x_2.$

Due to the Lemma \ref{MB-Lemma1} we can compute,  for any 
$(x_1,y_1), (x_2,y_2)\in M^\perp\times M^{\perp\perp}$, 
$f(x_1,y_1)\cdot f(x_2,y_2)=(x_1\wedge y_1)\cdot(x_2\wedge y_2)=%
(x_1\cdot y_1)\cdot (x_2\cdot y_2)=x_1\cdot (y_1\cdot x_2)\cdot y_2= %
x_1\cdot x_2\cdot y_1\cdot y_2=(x_1\cdot x_2)\wedge (y_1\cdot y_2)=f(x_1\cdot x_2,y_1\cdot y_2).$

Moreover,  for any 
$(x_1,y_1), (x_2,y_2)\in M^\perp\times M^{\perp\perp}$ and any $z\in A$, 
$z\leq (x_1\cdot y_1)\to x_2$ iff $z\cdot x_1\cdot y_1\leq x_2$ iff 
 $z\cdot x_1\leq y_1\to x_2$ iff (by Lemma \ref{MB-Lemma1}) 
$z\cdot x_1\leq x_2$ iff $z\leq x_1\to x_2$. It follows that 
$(x_1\cdot y_1)\to x_2=x_1\to x_2$. Similarly, 
$(y_1\cdot x_1)\to y_2=y_1\to y_2$. 

This yields, for  any 
$(x_1,y_1), (x_2,y_2)\in M^\perp\times M^{\perp\perp}$ and any $z\in A$, 
$z\leq f(x_1,y_1)\to f(x_2,y_2)$ iff 
$z\cdot x_1\cdot y_1\leq x_2\wedge y_2$ iff 
$z\cdot x_1\cdot y_1\leq x_2$ and $z\cdot x_1\cdot y_1\leq y_2$ iff 
$z\leq x_1\cdot y_1\to x_2$ and $z\leq x_1\cdot y_1\to y_2$ iff 
$z\leq x_1\to x_2$ and $z\leq y_1\to y_2$ iff 
$z\leq (x_1\to x_2)\wedge  (y_1\to y_2)$. Since  
$x_1\to x_2\in M^\perp$ and $y_1\to y_2\in  M^{\perp\perp}$ we get that 
$f(x_1,y_1)\to f(x_2,y_2)= f(x_1\to x_2,y_1\to y_2)$. 
Analogously, we can prove $f(x_1,y_1)\leadsto f(x_2,y_2)=f(x_1\leadsto x_2,y_1\leadsto y_2).$ 
\end{proof}

The previous Theorem shows that $M^\perp\cdot M^{\perp\perp}$ is both a filter and a sub hoop 
of  $\mathbf A$. Moreover, the sub hoop $M^\perp\cdot M^{\perp\perp}$ is directly reducible. 
The idea of our research leads us to decide whether some $a\in A$ belongs to the sub 
hoop $M^\perp\cdot M^{\perp\perp}$ or not. If $a\in M^\perp\cdot M^{\perp\perp}$ then $a=x\wedge y,$ 
where $x$ is minimal in $M^\perp$ with $a\leq x$ and analogously $y$ is minimal in $M^{\perp\perp}$ with $a\leq y.$ 
Those facts are motivation for following definition. 

\begin{definition}
Let $F$ be any filter in a pseudo-hoop $\mathbf A=(A;\cdot,\to,\leadsto,1).$ Then for any $X\subseteq A$ we define $\nu_F(X):=\{a\in F\mid x\leq a\mbox{ for any }x\in X\}.$ If there is the least element of the set $\nu_F(X)$ then we denote 
it by $\hat\nu_F(X).$ 
\end{definition}

Note that,  for any $X\subseteq A$,  $\nu_F(X)$ is a set of upper bounds of $X$ that are in $F$. 
It follows that $\nu_F(X)=\mu_F(\{z\in A \mid z \ \text{is an upper bound of}\ X\}$. 
In particular, 
$\nu_F(X)$ is an upper set such that $1\in \nu_F(X)$.

In what follows  we denote for any subsets $X,Y\subseteq A$, 
where $\mathbf A=(A;\cdot,\to,\leadsto,1)$ is a pseudo-hoop, the following sets
$$
\begin{array}{r  c l}
X{\bm{\cdot}} Y &=& \{x\cdot y\mid x\in X, y\in Y\},\\[0.1cm]
X{\bm{\to}} Y &=& \{x\to y\mid x\in X, y\in Y\},\\[0.1cm]
X{\bm{\leadsto}} Y &=& \{x\leadsto y\mid x\in X, y\in Y\}.
\end{array}
$$

\begin{theorem}
Let $F$ be a filter in the pseudo-hoop $\mathbf A=(A;\cdot,\to,\leadsto,1)$ and let  $X\subseteq A$. 
Then the sets $\nu_F(\nu_F(X)\bmto X)$ and $\nu_F(\nu_F(X)\bmleadsto X)$ are a filters.

If $\hat\nu_F(x)$ exists for some $x\in A$ then also $\hat\nu_F(\hat\nu_F(x)\to x)$  and 
$\hat\nu_F(\hat\nu_F(x)\leadsto x)$ exist 
and moreover $\hat\nu_F(\nu_F(x)\to x)$ and  $\hat\nu_F(\nu_F(x)\leadsto x)$  are idempotent.
\end{theorem}
\begin{proof} It is enough to verify the statement for the implication $\bmto$. The remaining case 
for $\bmleadsto $ can be shown dually. 
Now, let us prove first that 
$$\nu_F(\nu_F(X)\bmto X)\bmdot\nu_F(X)=\nu_F(X)$$ for any $X\subseteq A$. Clearly, 
we have  $1\in\nu_F(\nu_F(X)\bmto X)$ and thus 
$\nu_F(\nu_F(X)\bmto X)\bmdot\nu_F(X)\supseteq\nu_F(X)$ holds.

Conversely, assume that  $a\in \nu_F(\nu_F(X)\bmto X)$ and $b\in\nu_F(X).$ Then 
$a, b\in F$ and 
$ b\to x \leq a$ for all $x\in X$. Consequently 
$x=x\wedge b = (b\to x)\cdot b\leq a\cdot b$ for all $x\in X$. 
Because $a\cdot b\in F$ (both $a$ and $b$ belong to $F$) also $a\cdot b\in \nu_F(X)$. 
Hence  $\nu_F(\nu_F(X)\bmto X)\bmdot \nu_F(X)\subseteq \nu_F(X)$.
Together we obtain $\nu_F(\nu_F(X)\bmto X)\bmdot \nu_F(X)= \nu_F(X)$.

One can easily check that the set equality $A\bmto(B\bmto C)=(A\bmdot B)\bmto C$ holds.
Thus we can compute
\begin{eqnarray*}
\nu_F(\nu_F(X)\bmto X)\bmto (\nu_F(X)\bmto X) &=&(\nu_F(\nu_F(X)\bmto X)\bmdot \nu_F(X))\bmto X\\
&=&\nu_F(X)\bmto X.
\end{eqnarray*}

Denoting $Y:=\nu_F(X)\bmto X$ in the previous equality we obtain
$$\nu_F(\nu_F(Y)\bmto Y)=\nu_F(Y),$$
which together with $\nu_F(\nu_F(Y)\bmto Y)\bmdot \nu_F(Y)=\nu_F(Y)$ gives 
$$\nu_F(Y)\bmdot\nu_F(Y)=\nu_F(Y).$$
Thus $\nu_F(Y)=\nu_F(\nu_F(X)\bmto X)$ is a filter.

If the element $\hat\nu_F(x)$ exists and if $x\leq y$ then 
clearly $y\leq y\vee \hat\nu_F(x)\in F$. Moreover, any $a\in F$ such that $y\leq a$ 
satisfies $x\leq a$ and consequently $\hat\nu_F(x)\leq a.$ Thus 
$y\vee \hat\nu_F(x)\leq y\vee a=a$ holds and we have 
proved that $\hat\nu_F(y)$ exists and that $\hat\nu_F(y)=y\vee\hat\nu_F(x).$

Because $x\leq \hat\nu_F(x)\to x$ then $\hat\nu_F(\hat\nu_F(x)\to x)=(\hat\nu_F(x)\to x) \vee\hat\nu_F(x)$ 
exists and it is the least element of the filter $\nu_F(\nu_F(\{x\})\bmto \{x\}).$ 
Thus $\hat\nu_F(\hat\nu_F(x)\to x)$ is an idempotent element.
\end{proof}

\begin{lemma}\label{mb_lemma1}
Let $\mathbf A=(A;\cdot,\to,\leadsto,1)$ be a pseudo-hoop and let $F\subseteq A$ be 
a filter with a least element then $F$ is a normal filter.
\end{lemma}
\begin{proof}
If $a=\bigwedge F$ is the least element of the filter $F$ 
then it is an idempotent element. Using the divisibility we obtain for any $x\in A$ 
\begin{eqnarray*}
x\cdot a &=& (x\cdot a)\wedge a \\
&=& a\cdot(a\leadsto (x\cdot a) )\\
&=& a\cdot a\cdot(a\leadsto (x\cdot a) )\\
&=& a\cdot (a\wedge (x\cdot a) )\\
&=& a\cdot x\cdot a.
\end{eqnarray*}
and
\begin{eqnarray*}
a\cdot x &=& (a\cdot x)\wedge a \\
&=& (a\to (a\cdot x) )\cdot a\\
&=& (a\to (a\cdot x) )\cdot a\cdot a\\
&=&  (a\wedge (a\cdot x) )\cdot a\\
&=& a\cdot x\cdot a.
\end{eqnarray*}

Hence $a\cdot x=x\cdot a$ holds. 

Finally, $x\to y\in F$ if, and only if, $a\leq x\to y$ if, and only if, $x\cdot a= a\cdot x\leq y$ if, and only if, $a\leq x \leadsto  y$ if, and only if, $x\leadsto y\in F$. Thus $F$ is a normal.

\end{proof}

\begin{theorem}
Let $\mathbf A=(A;\cdot,\to,\leadsto,1)$  be a pseudo-hoop and let $M\subseteq A.$ If there is 
an element $x\not\in M^\perp\cdot M^{\perp\perp}$ such that $\hat\nu_{M^\perp\cdot M^{\perp\perp}}(x)$ exists then $\mathbf A$ is subdirectly reducible.
\end{theorem}
\begin{proof}
Let us assume that the element $\hat\nu_{M^\perp\cdot M^{\perp\perp}}(x)$ exists. Then we denote the idempotent 
$$y:=\hat\nu_{M^\perp\cdot M^{\perp\perp}}(\hat\nu_{M^\perp\cdot M^{\perp\perp}}(x)\to x).$$
 We will show that $y\not\in M^{\perp},M^{\perp\perp}.$ 

Assume to the contrary that $y\in M^\perp$. Then for any $a\in M^{\perp\perp}$ we have $a\vee (\nu_{M^\perp\cdot M^{\perp\perp}}(x)\to x) =a\vee y= 1$ and thus $\nu_{M^\perp\cdot M^{\perp\perp}}(x)\to x\in M^{\perp}\subseteq {M^\perp\cdot M^{\perp\perp}}.$ Clearly, also $\nu_{M^\perp\cdot M^{\perp\perp}}(x)\in{M^\perp\cdot M^{\perp\perp}}$ and thus $(\nu_{M^\perp\cdot M^{\perp\perp}}(x)\to x)\cdot \nu_{M^\perp\cdot M^{\perp\perp}}(x)\leq x\in {M^\perp\cdot M^{\perp\perp}}$ which is absurd. Analogously it can be proved that $y\not\in M^{\perp\perp}.$

Because $y\in M^\perp\cdot M^{\perp\perp},$ there exist elements 
$1\not=y_1\in M^\perp$ and $1\not=y_2\in M^{\perp\perp}$ such that 
$y=y_1\wedge y_2$ (see Theorem \ref{mb_theor1}). Clearly, Theorem \ref{mb_theor1} shows 
that $\langle y_1,y_2\rangle$ is an idempotent element in $M^\perp\times M^{\perp\perp}$ 
(because $y$ is an idempotent in $ M^\perp\cdot M^{\perp\perp}$). Consequently, both 
$y_1$ and $y_2$ are idempotent too.

Lemma \ref{mb_lemma1} shows that $\mathcal{F}(y_1)$ and $\mathcal{F}(y_2)$ are a normal filters. 
Moreover $y_1\vee y_2=1$ yields $\mathcal{F}(y_1)\cap \mathcal{F}(y_2)=\{ 1\}$ and thus 
$\mathbf A$ is subdiretly reducible. 
\end{proof}

\section{Several types of prime filters in residuated $\vee$-semilattices}
\label{filterkondo}

In the paper \cite{gasse} Van Gasse et al. asked whether, for any commutative residuated lattice
$L$, if prime filters and $\vee$-prime filters coincide,
then $L$ must be an MTL-algebra. The  affirmative answer was given 
in \cite{kondotur} by Kondo and Turunen. We extend their result to the setting 
of noncommutative residuated semilattices. 

\begin{definition}\label{leadstfilt}
Let $A$ be a residuated $\vee$-semilattice. A filter $F$  of $A$ is called a
{\em $\to$-prime filter} ({\em $\leadsto$-prime filter}) if $x \to y \in F$  or 
$y \to x \in F$ ($x \leadsto y \in F$  or 
$y \leadsto x \in F$) for all $x, y \in A$. A filter $F$  of $A$ is said to be a
{\em prime filter} if it is both a $\to$-prime filter and a $\leadsto$-prime filter. A filter $F$  of $A$ is called a {\em $\vee$-prime filter} if 
$x \vee y \in F$ yields $x \in F$ or $y \in F$ for all $x, y \in A$. 
By ${\mathcal P}{\mathcal F}(A)$ (${\mathcal P}{\mathcal F}_{\to}(A)$, 
${\mathcal P}{\mathcal F}_{\leadsto}(A)$, 
${\mathcal P}{\mathcal F}_{\vee}(A)$, respectively) 
we mean the class of all prime filters
of L ($\to$-prime filters, $\leadsto$-prime filters, $\vee$-prime filters, respectively). 
\end{definition}

\begin{lemma} \label{bcxd} Let $A$ be an integral residuated $\vee$-semilattice. Then 
\begin{enumerate}
\item ${\mathcal P}{\mathcal F}_{\to}(A)\subseteq 
{\mathcal P}{\mathcal F}_{\vee}(A)$; 
\item ${\mathcal P}{\mathcal F}_{\leadsto}(A)\subseteq 
{\mathcal P}{\mathcal F}_{\vee}(A)$; 
\item ${\mathcal P}{\mathcal F}(A)\subseteq 
{\mathcal P}{\mathcal F}_{\vee}(A)$.
\end{enumerate}
\end{lemma}
\begin{proof} 1. Let $F\in{\mathcal P}{\mathcal F}_{\to}(A)$, 
$x, y\in A$, $x\vee y\in F$. Then $x \to y \in F$  or 
$y \to x \in F$. Assume first that $x \to y \in F$. Clearly, 
$1=y \to y \in F$. Hence also 
$(x \to y)\cdot (y \to y)\leq (x \to y)\wedge (y \to y)=(x\vee y) \to y \in F$. 
It follows that $((x\vee y) \to y)\cdot (x\vee y)\leq y \in F$. Similarly, 
$y \to x \in F$ yields that $x\in F$. 

\noindent{}2. It follows by corresponding reasonings as in (1) applied to 
$\leadsto$. 

\noindent{}3. It follows from the fact that 
${\mathcal P}{\mathcal F}(A)=%
{\mathcal P}{\mathcal F}_{\to}(A)\cap%
{\mathcal P}{\mathcal F}_{\leadsto}(A)$.
\end{proof}

\begin{definition}\label{mtlleadstfilt}
Let $A$ be an integral residuated $\vee$-semilattice. $A$ is called a 
{\em $\to$-MTL-algebra} (a {\em $\leadsto$-MTL-algebra}, respectively) if 
$(x\to y)\vee (y\to x)=1$  ($(x\leadsto y)\vee (y\leadsto x)=1$, respectively)
for all $x, y\in A$. $A$ is said to be a {\em pseudo MTL-algebra} if  $A$ is both 
 a {$\to$-MTL-algebra} and a {$\leadsto$-MTL-algebra}.
\end{definition}

\begin{lemma} \label{cxd} Let $A$ be an integral residuated $\vee$-semilattice. Then 
\begin{enumerate}
\item If $A$ is a 
{$\to$-MTL-algebra} then  ${\mathcal P}{\mathcal F}_{\to}(A)= 
{\mathcal P}{\mathcal F}_{\vee}(A)$. 
\item If $A$ is a 
{$\leadsto$-MTL-algebra} then  
${\mathcal P}{\mathcal F}_{\leadsto}(A)=
{\mathcal P}{\mathcal F}_{\vee}(A)$; 
\item If $A$ is a pseudo {MTL-algebra} then ${\mathcal P}{\mathcal F}(A)=
{\mathcal P}{\mathcal F}_{\vee}(A)$.
\end{enumerate}
\end{lemma}
\begin{proof} 1. Let $F\in{\mathcal P}{\mathcal F}_{\vee}(A)$, 
$x, y\in A$. Since $(x\to y)\vee (y\to x)=1\in F$ we have that 
$x \to y \in F$  or $y \to x \in F$. It follows that 
$F\in{\mathcal P}{\mathcal F}_{\to}(A)$.

\noindent{}2. It follows by corresponding reasonings as in (1) applied to 
$\leadsto$. 

\noindent{}3.  Let $A$ be  a pseudo {MTL-algebra}. Then we have by (1) and (2) that  
${\mathcal P}{\mathcal F}_{\to}(A)= {\mathcal P}{\mathcal F}_{\vee}(A)$ 
and ${\mathcal P}{\mathcal F}_{\leadsto}(A)={\mathcal P}{\mathcal F}_{\vee}(A)$. 
It follows that ${\mathcal P}{\mathcal F}_{\leadsto}(A)={\mathcal P}{\mathcal F}_{\to}(A)=%
{\mathcal P}{\mathcal F}(A)={\mathcal P}{\mathcal F}_{\vee}(A)$.
\end{proof}

\begin{theorem}  \label{oddelth}
{\rm{}(The Prime filter theorem for 2-sided residuated $\vee$-semilattices).}  Let $A$ be 
a 2-sided residuated 
$\vee$-semilattice, $F$ be a filter of $A$ and $a\notin  $F. Then there exists a $\vee$-prime filter  $G$  of 
$A$ that includes $F$ and does not contain $a$.
\end{theorem}
\begin{proof} Let $G$ be a maximal filter containing $F$ that does not contain $a$. Let us show that 
$G$ is $\vee$-prime. Assume that  $x, y\in A$, $x\vee y\in G$ and $x, y\notin G$. Then, by 
maximality of $G$, we get that 
$$
\begin{array}{r c l l}
a\in [G\cup \{x\})&=&
 \{z \in A \mid& z\geq  u_1 \cdot x\cdot u_2 \cdot x \cdot \dots \cdot x \cdot  u_n\ \text{for some}\  n\in {\mathbb N},  n\geq 1\\ 
& & &\text{and}\  u_1, u_2, \dots, u_n \in  G\}.\\
\end{array}
$$
\noindent 
and 
$$
\begin{array}{r c l l}
a\in [G\cup \{y\})&=&
 \{z \in A \mid& z\geq  v_1 \cdot y\cdot v_2 \cdot y \cdot \dots \cdot y \cdot  v_m\ \text{for some}\  m\in {\mathbb N},  m\geq 1\\ 
& & &\text{and}\  v_1, v_2, \dots, v_m \in  G\}.\\
\end{array}
$$
This yields that there are $m, n\in {\mathbb N}, m,  n\geq 1$ and $u_1, u_2, \dots, u_n, v_1, v_2, \dots, v_m  \in  G$ such that 
$ u_1 \cdot x\cdot u_2 \cdot x \cdot \dots \cdot x \cdot  u_n\leq a$ and 
$v_1 \cdot y\cdot v_2 \cdot y \cdot \dots \cdot y \cdot  v_m\leq a$. Let us put 
$k=\mathrm{max} \{m, n\}$  
and  $w=\prod_{1\leq i\leq n, 1\leq j\leq m}u_i\cdot v_j$. Then 
$w\in G$. We also put  $w_i=w$ for all $i\leq k$.

Evidently, $ (wx)^{k}=w_1 \cdot x\cdot w_2 \cdot x \cdot \dots \cdot x \cdot  w_k\leq u_1 \cdot x\cdot u_2 \cdot x \cdot \dots \cdot x \cdot  u_n\leq a$ 
and $ (wy)^{k}=w_1 \cdot y\cdot w_2 \cdot y \cdot \dots \cdot y \cdot  w_k\leq v_1 \cdot y\cdot v_2 \cdot y \cdot \dots \cdot y \cdot  v_m\leq a$.

Let us compute the element 
$c=[w(x\vee y)]^{2k}\in G$. First, for any subset 
$S\subseteq \{1, \dots, 2k\}$ we put 
$c_S=\prod_{1\leq i\leq 2k} (w\cdot z_i)$, where $z_i=x$ if $i\in S$ and $z_i=y$ otherwise. Clearly, 
if $\mathrm{card}(S)\geq k$ then $c_S\leq (wx)^{k}\leq a$ and similarly if  
$\mathrm{card}(S)\leq k$ then $c_S\leq (wy)^{k}\leq a$. Since 
$c=\bigvee_{S\subseteq \{1, \dots, 2k\}}c_S$ we get that $c\leq a$, i.e. $a\in G$, a contradiction.
\end{proof}
We  then have the following corollary.

\begin{corollary}\label{inters} Any filter of a 2-sided residuated  $\vee$-semilattice is equal to the intersection of
the prime filters that include it.
\end{corollary}

\begin{theorem} \label{bacxd} Let $A$ be an integral residuated $\vee$-semilattice. Then 
\begin{enumerate}
\item $A$ is a 
{$\to$-MTL-algebra} if and only if   ${\mathcal P}{\mathcal F}_{\to}(A)= 
{\mathcal P}{\mathcal F}_{\vee}(A)$. 
\item $A$ is a 
{$\leadsto$-MTL-algebra}  if and only if 
${\mathcal P}{\mathcal F}_{\leadsto}(A)=
{\mathcal P}{\mathcal F}_{\vee}(A)$; 
\item $A$ is a {pseudo MTL-algebra}  if and only if  ${\mathcal P}{\mathcal F}(A)=
{\mathcal P}{\mathcal F}_{\vee}(A)$.
\end{enumerate}
\end{theorem}
\begin{proof} 1.  $\Longrightarrow$: It follows from Lemma \ref{cxd}, (i).  

$\Longleftarrow$: Assume that  ${\mathcal P}{\mathcal F}_{\to}(A)= 
{\mathcal P}{\mathcal F}_{\vee}(A)$ and  that  $A$ is not  a 
{$\to$-MTL-algebra}. Hence there are $a, b\in A$ such that 
$(a\to b)\vee (b\to a)\not =1$. Let 
$G_1=\bigcap \{ G\in {\mathcal F}{(A)} \mid G\not=\{1\} \} $. Assume first that 
$G_1=\{1\}$. Then there exists a $\vee$-prime filter $P$ such that 
$(a\to b)\vee (b\to a)\not \in P$. Since $P$ is also a  $\to$-prime filter we have that 
$(a\to b)\in P$  or $ (b\to a) \in P$, i.e., 
$(a\to b)\vee (b\to a)\in P$ which yields a contradiction.

Second, assume that $G_1\not=\{1\}$. Then $\{1\}$ is a $\vee$-prime filter, 
hence a  $\to$-prime filter. It follows that either $1\leq a\to b$ or 
$1\leq b\to a$, i.e., $(a\to b)\vee (b\to a) =1$, a contradiction again.

2. It follows by the same arguments as in part 1.

3. It follows from parts 1 and 2.
\end{proof}

\section*{Acknowledgment}

This is a pre-print of an article published in International Journal of Theoretical Physics. 
The final authenticated version of the article is available online at: https://doi.org/10.1007/s10773-015-2608-0.

\end{document}